\documentclass[12pt,reqno]{amsart}
\usepackage{amssymb,amsthm}
\usepackage{amsmath,amssymb}
\newtheorem{thm}{Theorem}
\newtheorem{lem}{Lemma}
\newtheorem{rmk}{Remark}
\numberwithin{equation}{section}

\def\BC{\mathbb C}
\def\BN{\mathbb N}
\def\BR{\mathbb R}
\def\cB{\mathcal B}
\def\cD{\mathcal D}
\def\rd{\mathrm d}
\def\Ga{\Gamma}
\def\La{\Lambda}
\def\Om{\Omega}
\def\al{\alpha}
\def\be{\beta}
\def\ga{\gamma}
\def\ve{\varepsilon}
\def\te{\theta}

\def\la{\lambda}
\def\vp{\varphi}
\def\om{\omega}
\def\f{\frac}
\def\pa{\partial}
\def\tri{\triangle}
\def\un{\underline}

\allowdisplaybreaks[4]

\renewcommand{\labelenumi}{(\theenumi)}

\title[Blowup and global existence for time-fractional diffusion equations]
{Blowup in $L^1(\Om)$-norm and global existence\\
for time-fractional diffusion equations\\
with polynomial semilinear terms}

\author{Giuseppe Floridia$^1$, Yikan Liu$^2$ and 
Masahiro Yamamoto$^3$}

\thanks{
$^1$ Department of Basic and Applied Sciences for Engineering,
Sapienza Universit\`a di Roma, Via Antonio Scarpa 16,
00161 Roma, Italy. E-mail: {\tt giuseppe.floridia@uniroma1.it}\\
$^2$ Department of Mathematics, Kyoto University,
Kitashirakawa-Oiwakecho, Sakyo-ku, Kyoto 606-8502, Japan.
E-mail: {\tt liu.yikan.8z@kyoto-u.ac.jp}\\
$^3$ Graduate School of Mathematical Sciences,
The University of Tokyo, 3-8-1 Komaba, Meguro-ku,
Tokyo 153-8914, Japan;
Honorary Member of Academy of Romanian Scientists,
Ilfov, Nr. 3, Bucuresti, Romania;
Correspondence member of Accademia Peloritana dei Pericolanti,
Palazzo Universit\`a, Piazza S. Pugliatti 1, 98122 Messina, Italy.
E-mail: {\tt myama@ms.u-tokyo.ac.jp}}

\date{}

\begin{document}

\baselineskip=15pt
\maketitle

\begin{abstract}
This article is concerned with semilinear time-fractional diffusion equations
with polynomial nonlinearity $u^p$ in a bounded domain $\Om$
with the homogeneous Neumann boundary condition and positive initial values.
In the case of $p>1$, we prove the blowup of solutions $u(x,t)$ in the sense 
that $\|u(\,\cdot\,,t)\|_{L^1(\Om)}$ tends to $\infty$ as $t$ approaches
some value, 
by using a comparison principle for the corresponding ordinary differential 
equations
and constructing special lower solutions. Moreover, we provide 
an upper bound for the blowup time.
In the case of $0<p<1$, we establish the global existence of solutions in time
based on the Schauder fixed-point theorem.\medskip

\medskip

\end{abstract}

\section{Introduction and Main Results}

Let $d=1,2,3$ and $\Om\subset\BR^d$ be a bounded domain
with smooth boundary $\pa\Om$. For $0<\al<1$,
let $\rd_t^\al$ denote the classical Caputo derivative:
$$
\rd_t^\al f(t):=\int^t_0\f{(t-s)^{-\al}}{\Ga(1-\al)}f'(s)\,\rd s,
\quad f\in W^{1,1}(0,T).
$$
Here, $\Ga(\,\cdot\,)$ denotes the gamma function.

For consistent discussions of semilinear time-fractional
diffusion equations, we extend the classic Caputo derivative
$\rd_t^\al$ as follows. First, for $0<\al<1$,
we define the Sobolev-Slobodecki space $H^\al(0,T)$ 
with the norm $\|\cdot\|_{H^\al(0,T)}$ as follows:
$$
\|f\|_{H^\al(0,T)}:=\left(\|f\|^2_{L^2(0,T)}+
\int^T_0\!\!\!\int^T_0\f{|f(t)-f(s)|^2}{|t-s|^{1+2\al}}\,\rd t\rd s\right)^{\f12}
$$
(e.g., Adams \cite{Ad}). Furthermore, we set $H^0(0,T):=L^2(0,T)$ and
$$
H_\al(0,T):=
\left\{\!\begin{alignedat}2
& H^\al(0,T), &\quad & 0<\al<\f12,\\
&\left\{f\in H^{\f12}(0,T);
\int^T_0\f{|f(t)|^2}t\,\rd t<\infty\right\},
&\quad & \al=\f12,\\
& \{f\in H^\al(0,T);f(0)=0\}, &\quad &\f12<\al\le1
\end{alignedat}\right.
$$
with the norms defined by 
$$
\|f\|_{H_\al(0,T)}:=
\left\{\!\begin{alignedat}2
& \|f\|_{H^\al(0,T)}, &\quad & \al\ne\f12,\\
& \left(\|f\|^2_{H^{\f12}(0,T)}+\int^T_0\f{|f(t)|^2}t\,\rd t\right)^{\f12}, 
& \quad &\al=\f12.
\end{alignedat}\right.
$$
Moreover, for $\be>0$, we set 
$$
J^\be f(t):=\int^t_0\f{(t-s)^{\be-1}}{\Ga(\be)}f(s)\,\rd s,
\quad0<t<T,\ f\in L^1(0,T).
$$
Then, it was proved e.g., in the study by Gorenflo et al.\! \cite{GLY}, that 
$J^\al:L^2(0,T)\longrightarrow H_\al(0,T)$ is an isomorphism for $\al\in(0,1)$.

Now we are ready to define the extended Caputo derivative
$$
\pa_t^\al:=(J^\al)^{-1},\quad\cD(\pa_t^\al)=H_\al(0,T).
$$
Henceforth, $\cD(\,\cdot\,)$ denotes the domain of an operator under consideration.
This is the minimum closed extension of $\rd_t^\al$ with
$\cD(\rd_t^\al):=\{v\in C^1[0,T];v(0)=0\}$ and
$\pa_t^\al v=\rd_t^\al v$ for $v\in C^1[0,T]$ satisfying
$v(0)=0$. As for the details, we can refer to the studies by
Gorenflo et al.\! \cite{GLY} and Yamamoto \cite{Y2023}.

This article is concerned with the following initial-boundary value problem for a nonlinear time-fractional diffusion equation:
\begin{equation}\label{1.1}
\begin{cases}
\pa_t^\al(u-a)=\tri u+u^p & \mbox{in }\Om\times(0,T),\\
\pa_\nu u=0 & \mbox{on }\pa\Om\times(0,T),
\end{cases}
\end{equation}
where $p>0$ is a constant.
The left-hand side of the time-fractional differential equation
in equation \eqref{1.1} means that $u(x,\,\cdot\,)-a(x)\in H_\al(0,T)$ for almost all
$x\in\Om$. For $\f12<\al<1$, since
$v\in H_\al(0,T)$ implies $v(0)=0$ by the trace theorem, we can
understand that the left-hand side means that $u(x,0)=a(x)$
in the trace sense with respect to $t$. As a result, this corresponds to the
initial condition for $\al>\f12$, whereas we do not need any
initial conditions for $\al<\f12$.

There are other formulations for initial-boundary value problems
for time-fractional partial differential equations
(e.g., Sakamoto and Yamamoto \cite{SY} and Zacher \cite{Za}), but
here we do not provide comprehensive references.
In the case of $\al=1$, concerning the non-existence of
global solutions in time, there have been enormous works since
Fujita \cite{F}, and we can refer to a comprehensive
monograph by Quittner and Souplet \cite{QS}.
We can refer to Fujishima and Ishige \cite{FI} and 
Ishige and Yagisita \cite{IshiY}
as related results to our first main result Theorem~\ref{thm1} stated below.
See also Chen and Tang \cite{CT22}, Du \cite{D22},
Feng et al.\! \cite{FQZC23}, and Tian and Xiang \cite{TX23}.

For $0<\al<1$, the time-fractional diffusion equation in \eqref{1.1} is
a possible model for describing anomalous diffusion
in heterogeneous media, and the semilinear term $u^p$ can describe
a reaction term.
There are also rapidly increasing interests for the non-existence
of global solutions to semilinear time-fractional differential
equations such as equation \eqref{1.1}. As recent works, we refer to
studies by Ahmad et al.\! \cite{AAAKT},
Borikhanov et al.\! \cite{BRT23},
Ghergu et al.\! \cite{GMS23},
Hnaien et al.\! \cite{HKL},
Kirane et al.\! \cite{KLT}, Kojima \cite{K23},
Suzuki \cite{S21,S22}, Vergara and Zacher \cite{VZ},
and Zhang and Sun \cite{ZS15}.
In \cite{VZ} and \cite{ZS15}, the blowup is considered by
$\|u(\,\cdot\,,t)\|_{L^1(\Om)}$.  
Since $L^1(\Om)$-norm is 
the weakest among the Lebesgue space norms, the choice $L^1(\Om)$ 
as spatial norm is sharp for consideration of the blowup.

Our approach is based on the comparison of solutions to
initial value problems for time-fractional ordinary
differential equations, which is similar to that by Ahmad et al.\! \cite{AAAKT} in the sense that
the scalar product of the solution with the first eigenfunction
of the Laplacian with the boundary condition is considered.
Vergara and Zacher, in their study \cite{VZ}, discuss stability, instability, and
blowup for time-fractional diffusion equations
with super-linear convex semilinear terms.

To the best knowledge of the authors, there are no publications
providing an upper bound of the blowup time for the time-fractional 
diffusion equation in $L^1(\Om)$-norm, which is weaker than 
$L^q(\Om)$-norm with $1<q\le \infty$.

Throughout this article, we assume $\f34<\ga\le1$. First, for $p>1$,
we recall a basic result on the unique existence of local 
solutions in time. For $a\in H^{2\ga}(\Om)$ satisfying 
$\pa_\nu a=0$ on $\pa\Om$ and $a\ge0$ in $\Om$,
Luchko and Yamamoto \cite{LY} proved
the unique existence, which is local in time $t$.
There exists a  constant $T>0$ depending on $a$ such that
\eqref{1.1} possesses a unique solution $u$ such that 
\begin{equation}\label{1.2}
u\in C([0,T];H^{2\gamma}(\Om)),\quad u-a\in H_\al(0,T;L^2(\Om))
\end{equation}
and $u\ge0$ in $\Om\times(0,T)$. 
The time length $T$ of the existence of $u$ 
does not depend on the choice of initial values and
only depends on a bound $m_0>0$ such that
$\|a\|_{H^{2\ga}(\Om)}\le m_0$, provided that $\pa_\nu a=0$ on $\pa\Om$.

We call $T_{\al,p,a} > 0$ the blowup time 
in $L^1(\Om)$ of the solution to 
\eqref{1.1} if  
\begin{equation}\label{eq-blowup}
\lim_{t\uparrow T_{\al,p,a}}\|u(\,\cdot\,,t)\|_{L^1(\Om)}=\infty.
\end{equation}
As the non-existence of global solutions in time, in this article we are 
concerned with the blowup in $L^1(\Om)$.

Now we are ready to state our first main results on the blowup 
with an upper bound of the blowup time for $p>1$.

\begin{thm}\label{thm1}
Let $p>1$ and $a\in H^{2\ga}(\Om)$ satisfy $\pa_\nu a=0$ on $\pa\Om$
and $a\ge0,\not\equiv0$ in $\Om$. Then, there exists some
$T=T_{\al,p,a} >0$ such that the solution satisfying \eqref{1.2} 
exists for $0<t<T_{\al,p,a}$ and \eqref{eq-blowup} holds. 
Moreover, we can bound $T_{\al,p,a}$ from above as:
\begin{equation}\label{1.4}
T_{\al,p,a} \le \left( \f1
{(p-1)\Ga(2-\al)\left( \f1{| \Om|}
\int_{\Om} a(x)\, \rd x \right)^{p-1} } \right)^{\f1{\al}}
=: T^*(\al,p,a).
\end{equation}
\end{thm}

\begin{rmk}
{\rm (1) We note that $T^*(\al,p,a)$ decreases as $\int_\Om a(x)\,\rd x$ 
increases for arbitrarily fixed $p$ and $a$.
Meanwhile, $T^*(\al,p,a)$ tends to $\infty$ as $p>1$ approaches $1$, which is
consistent because $p=1$ is a linear case and we have no blowup.

(2) Estimate \eqref{1.4} corresponds to the estimate
in \cite[Remark 17.2(i) (p.105)]{QS} for $\al=1$.
On the other hand,
in the case of parabolic equations $\pa_tu = D\tri u + u^p$
with constant $D>0$, Ishige and Yagisita discussed
the asymptotics of the blowup time $T_{p,a}(D)$ and established
$$
T_{p,a}(D) = \f1{(p-1)\left( \f1{| \Om|}
\int_{\Om} a(x)\, \rd x \right)^{p-1} } + O\left( \f1{D}\right)
\quad\mbox{as $D \to \infty$}
$$
(\cite[Theorem 1.1]{IshiY}).
The principal term of the asymptotics coincides with 
the value obtained by substituting $\al=1$ in $T^*(\al,p,a)$ given by 
\eqref{1.4}. Thus, $T^*(\al,p,a)$ is not only one possible upper bound of
the blowup time for $0<\al<1$ but also seems to capture some essence.
Moreover, Ishige and Yagisita \cite{IshiY} clarifies the blowup set; see also 
the work of Fujishima and Ishige \cite{FI}. For $0<\al<1$, there are no such 
detailed available results.}
\end{rmk}

The second main result is the global existence of solutions to \eqref{1.1} 
for  $0<p<1$. 

\begin{thm}\label{thm2}
Let $0<p<1$ and $a\in H^{2\ga}(\Om)$ satisfy
$\pa_\nu a=0$ on $\pa\Om$ and $a\ge0$ in $\Om$.
For arbitrarily given $T>0,$ there exists a global solution $u$
to \eqref{1.1} with $T=\infty$ satisfying \eqref{1.2}.
\end{thm}

In Theorem \ref{thm2}, we cannot further conclude the uniqueness of the 
solution.
This is similar to the case of $\al=1$, where the uniqueness relies
essentially on the Lipschitz continuity of the semilinear term $u^p$
in $u\ge0$. Indeed, we can easily give a counterexample
by a time-fractional ordinary differential equation:
$$
\pa_t^\al y(t)=\f{\Ga(2\al+1)}{\Ga(\al+1)}y(t)^{\f12},
$$
where $y\in H_\al(0,T)$. Then, we can directly verify that both $y(t)=t^{2\al}$ and
$y(t)\equiv0$ are solutions to this initial value problem.

The key to the proof of Theorem \ref{thm1} is a comparison principle \cite{LY}
and a reduction to a time-fractional
ordinary differential equation. Such a reduction method can be found in the studies by
Kaplan \cite{Kp} and Payne \cite{Pa} for the case $\al=1$.
On the other hand, Theorem \ref{thm2} is proved by the Schauder fixed-point 
theorem with regularity properties of solutions \cite{Y2023}.
For a related method for Theorem \ref{thm2}, we refer to the study by
D\'iaz et al.\! \cite{DPV}.

This article is composed of five sections;
in Section \ref{sec-pre}, we show lemmata that complete the proof of Theorem \ref{thm1} in Section \ref{sec-pf1};
we prove Theorem \ref{thm2} in Section \ref{sec-pf2};
finally, Section \ref{sec-rem} is devoted to concluding remarks and discussions.

\section{Preliminaries}\label{sec-pre}

We will prove the following two lemmata.

\begin{lem}\label{lem1}
Let $f\in L^2(0,T)$ and $c\in C[0,T]$.
Then, there exists a unique solution $y\in H_\al(0,T)$ to
$$
\pa_t^\al y-c(t)y=f,\quad0<t<T.
$$
Moreover, if $f\ge0$ in $(0,T)$, then $y\ge0$ in $(0,T)$.
\end{lem}

\begin{proof}
The unique existence of $y$ is proved in
Kubica et al.\! \cite[Section 3.5]{KRY} for example.
The non-negativity $y\ge0$ in $(0,T)$ follows from the same
argument in the study by Luchko and Yamamoto \cite{LY}, which is based on the
extremum principle by Luchko \cite{L2009}.
\end{proof}

\begin{lem}\label{lem2}
Let $c_0>0,$ $a_0\ge0,$ $p>1$ be constants and
$y-a_0,z-a_0\in H_\al(0,T)\cap C[0,T]$ satisfy
$$
\pa_t^\al(y-a_0)\ge c_0y^p,\quad
\pa_t^\al(z-a_0)\le c_0z^p\quad\mbox{in }(0,T).
$$
Then, $y\ge z$ in $(0,T)$.
\end{lem}

\begin{proof}
We set
$$
\pa_t^\al(y-a_0)-c_0y^p=:f\ge0,\quad
\pa_t^\al(z-a_0)-c_0z^p=:g\le0.
$$
Since $y-a_0,z-a_0\in H_\al(0,T)\cap C[0,T]$,
we see that $f,g\in L^2(0,T)$. Setting
$$
\te:=y-z=(y-a_0)-(z-a_0)\in H_\al(0,T),
$$
we have
$$
\pa_t^\al\te	-c_0(y^p-z^p)=f-g\ge0\quad\mbox{in }(0,T).
$$
We can further prove that
\begin{equation}\label{2.1}
\pa_t^\al\te(t)-c_0c(t)\te(t)\ge0,\quad0<t<T,
\end{equation}
where
\begin{equation}\label{2.2}
c(t):=\left\{\!\begin{alignedat}2
& \f{y^p(t)-z^p(t)}{y(t)-z(t)}, & \quad & y(t)\ne z(t),\\
& p\,y^{p-1}(t), & \quad & y(t)=z(t).
\end{alignedat}\right.
\end{equation}
Indeed, we set $\La:=\{t\in [0,T];y(t)\ne z(t)\}$.
For $t_0\in\La$, we immediately see that 
$c(t_0)\te(t_0)=y(t_0)^p-z(t_0)^p$.
For $t_0 \not\in\La$, i.e.,
$$
\te(t_0)=(y(t_0)-a_0)-(z(t_0)-a_0)=0,
$$
first, we assume that there does not exist any sequence $\{t_n\}\subset\La$
such that $t_n\to t_0$. Then, there exists some small $\ve_0>0$ such that
$(t_0-\ve_0,t_0+\ve_0)\cap\La=\emptyset$. This means
$\te(t)=0$ for $t_0-\ve_0<t<t_0+\ve_0$, and thus,
$$
c(t)\te(t)=p\,y^{p-1}(t)\te(t)=0,\quad y^p(t)-z^p(t)=0,
\quad t_0-\ve_0<t<t_0+\ve_0.
$$
Hence, we obtain $c(t_0)\te(t_0)=y^p(t_0)-z^p(t_0)$.

Next, assume that there exists a sequence $\{t_n\}\subset\La$ such that 
$t_n\longrightarrow t_0\notin\La$ as $n\to\infty$.
By $t_n\in\La$, we have $y(t_n)\ne z(t_n)$ and 
$$
c(t_n)\te(t_n)=\f{y^p(t_n)-z^p(t_n)}{y(t_n)-z(t_n)}\te(t_n),
\quad n\in\BN.
$$
Since $y,z,\te\in C[0,T]$ and $\te(t_0)=0$,
we employ the mean value theorem to conclude
$$
\lim_{n\to\infty}\f{y^p(t_n)-z^p(t_n)}{y(t_n)-z(t_n)}\te(t_n)
=p\,y^{p-1}(t_0)\te(t_0)=0.
$$
Hence, again we arrive at $c(t_0)\te(t_0)=y^p(t_0)-z^p(t_0)$ in this case.
Thus, we have verified \eqref{2.1} with \eqref{2.2}.
Moreover, since $y,z\in C[0,T]$, we can verify that $c\in C[0,T]$.

Therefore, a direct application of Lemma \ref{lem1} to \eqref{2.1} yields
$\te\ge0$ in $(0,T)$ or equivalently $y\ge z$ in $(0,T)$.
Thus, the proof of Lemma \ref{lem2} is complete.
\end{proof}

\section{Completion of proof of Theorem \ref{thm1}}\label{sec-pf1}

{\bf Step 1.} We set 
$$
\eta(t):=\int_\Om u(x,t)\,\rd x=\int_\Om(u(x,t)-a(x))\,\rd x+a_0,\quad0<t<T,
$$
where $a_0:=\int_\Om a(x)\,\rd x$.
Here, we see that $a_0>0$ because $a\ge0,\not\equiv0$ in $\Om$ by the
assumption of Theorem \ref{thm1}.

\begin{rmk}\label{rmk1}
{\rm We note that $\eta(t)$ is the inner product of 
the solution $u(\,\cdot\,,t)$ with
the first eigenfunction $1$ of $-\tri$ with the homogeneous Neumann
boundary condition.  As for the parabolic case, we can refer to the studies by
Kaplan \cite{Kp} and Payne \cite{Pa}.}
\end{rmk}

Henceforth, we assume that the solution $u$ to \eqref{1.1} within the class 
\eqref{1.2} exists for $0<t<T$. 
By \eqref{1.2}, we have $\int_\Om(u(x,t)-a(x))\,\rd x\in H_\al(0,T)$.
Fixing $\ve>0$ arbitrarily small, we see
$$
\eta(t)-a_0=\int_\Om(u(x,t)-a(x))\,\rd x\in H_\al(0,T-\ve),
$$
and hence,
$$
\pa_t^\al(\eta(t)-a_0)=\int_\Om \pa_t^\al(u-a)(x,t)\,\rd x,\quad0<t<T-\ve.
$$
Since $\pa_\nu u=0$ on $\pa\Om\times(0,T-\ve)$, Green's formula and
the governing equation $\pa_t^\al(u-a)=\tri u+u^p$ yield
\begin{equation}\label{3.1}
\pa_t^\al(\eta(t)-a_0)=\int_\Om\tri u(x,t)\,\rd x+\int_\Om u^p(x,t)\,\rd x
=\int_\Om u^p(x,t)\,\rd x,\quad0<t<T-\ve.
\end{equation}
On the other hand, introducing the H\"older conjugate $q>1$ of $p>1$,
i.e., $\f1q+\f1p=1$, it follows from $u\ge0$ in $\Om\times(0,T-\ve)$
and the H\"older inequality that 
$$
\eta(t)=\int_\Om u(x,t)\,\rd x
\le\left(\int_\Om u^p(x,t)\,\rd x\right)^{\f1p}\left(\int_\Om\rd x\right)^{\f1q}
=|\Om|^{\f1q}\left(\int_\Om u^p(x,t)\,\rd x\right)^{\f1p},
$$
i.e., 
\begin{equation}\label{3.2}
\int_\Om u^p(x,t)\,\rd x\ge\om_0\,\eta^p(t),\quad\om_0:=|\Om|^{-\f p q}.
\end{equation}
By \eqref{3.1} and \eqref{3.2}, we obtain
\begin{equation}\label{3.3}
\pa_t^\al(\eta(t)-a_0)\ge\om_0\eta^p(t),\quad0<t<T-\ve.
\end{equation}

{\bf Step 2.} This step is devoted to the construction of
a lower solution $\un\eta(t)$ satisfying
\begin{equation}\label{eq-lower}
\pa_t^\al(\un\eta(t)-a_0)(t)\le\om_0\,\un\eta^p(t),\quad0<t<T-\ve,\quad\lim_{t\uparrow T}\un\eta(t)=\infty.
\end{equation}
We restrict the candidates of such a lower solution to
\begin{equation}\label{def-lowsol}
\un\eta(t):=a_0\left(\f T{T-t}\right)^m,\quad m\in\BN.
\end{equation}
To evaluate $\pa_t^\al(\un\eta(t)-a_0)(t)=\rd_t^\al\un\eta(t)$,
by definition, we have to represent $\f\rd{\rd t}(\f1{(T-t)^m})$
in terms of the Maclaurin expansion. First, direct calculations yield
\[
\f{\rd^m}{\rd t^m}\left(\f1{T-t}\right)=\f{m!}{(T-t)^{m+1}},
\]
and thus,
\begin{equation}\label{eq-calc0}
\f\rd{\rd t}\left(\f1{(T-t)^m}\right)=\f m{(T-t)^{m+1}}
=\f1{(m-1)!}\f{\rd^m}{\rd t^m}\left(\f1{T-t}\right).
\end{equation}
Next, by termwise differentiation, we have
$$
\f1{T-t}=\sum_{k=0}^\infty\f{t^k}{T^{k+1}},\quad
\f\rd{\rd t}\left(\f1{T-t}\right)
=\sum_{k=1}^\infty\f{kt^{k-1}}{T^{k+1}}
$$
for $0\le t\le T-\ve$. Repeating the calculations and by induction, we reach
\begin{equation}\label{eq-calc1}
\f{\rd^m}{\rd t^m}\left(\f1{T-t}\right)
=\sum_{k=m}^\infty\f{k(k-1)\cdots(k-m+1)}{T^{k+1}}t^{k-m}
=\f1{T^{m+1}}\sum_{k=0}^\infty\prod_{j=1}^m(k+j)\left(\f t T\right)^k.
\end{equation}
Plugging \eqref{eq-calc1} into \eqref{eq-calc0}, we obtain
\[
\f\rd{\rd t}\left(\f1{(T-t)^m}\right)=\f1{T^{m+1}(m-1)!}
\sum_{k=0}^\infty\prod_{j=1}^m(k+j)\left(\f t T\right)^k.
\]
Then, by the definition of $\rd_t^\al$, we calculate
\begin{align*}
\rd_t^\al\left(\f1{(T-t)^m}\right)
& =\int_0^t\f{(t-s)^{-\al}}{\Ga(1-\al)}\f\rd{\rd s}\left(\f1{(T-s)^m}\right)\rd s\\
& =\f1{\Ga(1-\al)T^{m+1}(m-1)!}\sum_{k=0}^\infty\f{\prod_{j=1}^m(k+j)}{T^k}\int_0^t(t-s)^{-\al}s^k\,\rd s.
\end{align*}
Here, we employ integration by substitution $s=t\xi$ and the beta function to treat
\begin{align*}
\int_0^t(t-s)^{-\al}s^k\,\rd s
& =t^{k+1-\al}\int_0^1(1-\xi)^{-\al}\xi^k\,\rd\xi\\
& =t^{k+1-\al}B(1-\al,k+1)=\f{\Ga(1-\al)\,k!}{\Ga(k+2-\al)}t^{k+1-\al},
\end{align*}
which implies
\begin{align*}
\rd_t^\al\left(\f1{(T-t)^m}\right)
& =\f1{T^{m+1}(m-1)!}\sum_{k=0}^\infty\f{\prod_{j=1}^m(k+j)\,k!}{\Ga(k+2-\al)}\f{t^{k+1-\al}}{T^k}\\
& =\f{t^{1-\al}}{T^{m+1}(m-1)!}\sum_{k=0}^\infty\f{(k+m)!}{\Ga(k+2-\al)}
\left(\f t T\right)^k.
\end{align*}
Since $\Ga(s)$ is monotone increasing in $s>2$ and $0<\Ga(2-\al)<1$,
for $k\in\BN\cup\{0\}$, we directly estimate
\[
\Ga(k+2-\al)\ge\left.\begin{cases}
\Ga(2-\al), & k=0,\\
\Ga(k+1)=k!, & k\in\BN
\end{cases}\right\}\ge\Ga(2-\al)\,k!.
\]
Then, we can bound $\rd_t^\al(\f1{(T-t)^m})$ from above as follows:
\begin{align*}
\rd_t^\al\left(\f1{(T-t)^m}\right) & \le\f{T^{1-\al}}{T^{m+1}(m-1)!}
\sum_{k=0}^\infty\f{(k+m)!}{\Ga(2-\al)\,k!}\left(\f t T\right)^k\\
& =\f1{\Ga(2-\al)\,T^{m+\al}(m-1)!}
\sum_{k=0}^\infty\prod_{j=1}^m(k+j)\left(\f t T\right)^k.
\end{align*}
For the series above, we utilize \eqref{eq-calc0} and \eqref{eq-calc1} again to 
find
\[
\f1{(T-t)^{m+1}}=\f1{m!}\f{\rd^m}{\rd t^m}\left(\f1{T-t}\right)
=\f1{T^{m+1}m!}\sum_{k=0}^\infty\prod_{j=1}^m(k+j)\left(\f t T\right)^k,
\]
indicating
\[
\rd_t^\al\left(\f1{(T-t)^m}\right)
\le\f1{\Ga(2-\al)\,T^{m+\al}(m-1)!}\f{T^{m+1}m!}{(T-t)^{m+1}}
= \f{T^{1-\al}m}{\Ga(2-\al)}\f1{(T-t)^{m+1}}.
\]
Recalling the definition \eqref{def-lowsol} of $\un\eta(t)$,
we eventually arrive at
\begin{equation}\label{eq-calc2}
\pa_t^\al(\un\eta(t)-a_0)=\rd_t^\al\un\eta(t)
=a_0T^m\rd_t^\al\left(\f1{(T-t)^m}\right)
\le\f{a_0T^{m+1-\al}m}{\Ga(2-\al)}\f1{(T-t)^{m+1}}.
\end{equation}
Note that \eqref{eq-calc2} holds for arbitrary $m\in\BN$, $T>0$, and $0<t<T-\ve$.

Finally, we claim that for any $p>1$ and $a_0>0$,
there exist constants $m\in\BN$ and $T>0$ such that
\begin{equation}\label{eq-calc3}
\f{a_0T^{m+1-\al}m}{\Ga(2-\al)}\f1{(T-t)^{m+1}}
\le\om_0\,\un\eta^p(t)=\f{\om_0a_0^pT^{m p}}{(T-t)^{m p}},
\quad0<t<T-\ve.
\end{equation}
In fact, \eqref{eq-calc3} is achieved by 
\[
\f{a_0T^{-\al}m}{\Ga(2-\al)} \le \om_0 a_0^p\left( \f{T}{T-t}\right)
^{mp-(m+1)} \quad \mbox{for $0<t<T$},
\]
which holds if 
\[
\f{a_0T^{-\al}m}{\Ga(2-\al)} \le \om_0 a_0^p
\]
by $mp-(m+1) \ge 0$ and $\f{T}{T-t} \ge 1$ for $0<t<T$.
Therefore, if 
\begin{align}
T & \ge\left(\f m{\Ga(2-\al)\,\om_0a_0^{p-1}}\right)^{\f1\al}
\ge\left(\f1{(p-1)\,\Ga(2-\al)\,\om_0a_0^{p-1}}\right)^{\f1\al}\nonumber\\
& =\left\{(p-1)\,\Ga(2-\al)\left(\f1{|\Om|}\int_\Om a\,\rd x\right)^{p-1}\right\}^{-\f1\al} =: T^*(\al,p,a),
\label{eq-lowT}
\end{align}
then \eqref{eq-calc3} is satisfied.
 
With the above chosen $m$ and $T^*(\al,p,a)$, consequently, it follows from
\eqref{eq-calc2} and \eqref{eq-calc3} that
$$
\un\eta(t)=a_0\left( \f{T^*(\al,p,a)}{T^*(\al,p,a)-t}\right)^m
$$
satisfies \eqref{eq-lower}.

Now it suffices to apply Lemma \ref{lem2} to
\eqref{eq-lower} and \eqref{3.3} on $[0,\,T^*(\al,p,a)-\ve]$ to obtain
$$
\eta(t)\ge\un\eta(t),\quad0\le t\le T^*(p,a)-\ve.
$$
Since $\ve>0$ was arbitrarily chosen, we obtain
$$
\int_\Om u(x,t)\,\rd x=\eta(t)\ge\un\eta(t)
= \f{a_0T^*(p,a)^m}{(T^*(p,a)-t)^m},\quad 0<t<T^*(\al,p,a).
$$
Since $\eta(t) = \|u(\,\cdot\,,t)\|_{L^1(\Om)}$, 
this means that the solution $u$ cannot exist for $t > T^*(\al,p,a)$.
Hence, the blowup time $T_{p,a} \le T^*(\al,p,a)$.
The proof of Theorem \ref{thm1} is complete.

\section{Proof of Theorem \ref{thm2}}\label{sec-pf2}

{\bf Step 1.} Henceforth, we denote the norm and the inner product of $L^2(\Om)$ by
$$
\|a\|:=\|a\|_{L^2(\Om)},\quad(a,b):=\int_\Om a(x)b(x)\,\rd x,
$$
respectively. We show the following lemma.

\begin{lem}\label{lem3}
Let $0<p<1,$ $0\le w\in L^2(\Om)$ and $0\le\eta\in L^2(0,T)$. Then,
$$
\|w^p\|\le|\Om|^{\f{1-p}2}\|w\|^p,\quad
\|\eta^p\|_{L^2(0,T)}\le T^{\f{1-p}2}\|\eta\|^p_{L^2(0,T)}.
$$
\end{lem}

\begin{proof}
By $0<p<1$, we see that $\f1{1-p}>1$, and the H\"older inequality yields
$$
\|w^p\|^2=\int_\Om w^{2p}\,\rd x
\le\left(\int_\Om(w^{2p})^{\f1p}\,\rd x\right)^p
\left(\int_\Om1^{\f1{1-p}}\,\rd x\right)^{1-p}
=|\Om|^{1-p}\|w\|^{2p},
$$
which completes the proof for $w$. The proof for $\eta$ is the same.
\end{proof}

Let $A=-\tri$ with $\cD(A)=\{w\in H^2(\Om);\pa_\nu w=0\mbox{ on }\pa\Om\}$.
We number all the eigenvalues of $A$ as
$$
0=\la_1\le\la_2\le\cdots,\quad\la_n\longrightarrow\infty\ (n\to\infty),
$$
with their multiplicities. By $\{\vp_n\}_{n\in\BN}$, we denote
the complete orthonormal basis of $L^2(\Om)$ formed by
the eigenfunctions of $A$, i.e., $A\vp_n=\la_n\vp_n$ and
$\|\vp_n\|=1$ for $n\in\BN$. We can define the fractional power
$A^\be$ for $\be\ge0$, and we know that
$\|a\|_{H^{2\be}(\Om)}\le C\|A^\be a\|$ for all $a\in\cD(A^\be)$,
where the constant $C>0$ depends on  $\be,\Om$
(e.g., \cite{LM,Paz}).
 
We further introduce the Mittag-Leffler functions by
$$
E_{\al,\be}(z)=\sum_{k=0}^\infty\f{z^k}{\Ga(\al k+\be)},
\quad z\in\BC,
$$
where $0<\al<1$ and $\be>0$. It is known that $E_{\al,\be}(z)$
is an entire function in $z\in\BC$, and we can refer, e.g., to
Podlubny \cite{Po} for further properties of $E_{\al,\be}(z)$.

Henceforth, we abbreviate $u(t)=u(\,\cdot\,,t)$ and interpret
$u(t)$ as a mapping from $(0,T)$ to $L^2(\Om)$. We define
\begin{align*}
S(t)a & :=\sum_{k=1}^\infty(a,\vp_n)E_{\al,1}(-\la_n t^\al)\vp_n,\\
K(t)a & :=\sum_{k=1}^\infty(a,\vp_n)t^{\al-1}E_{\al,\al}(-\la_n t^\al)\vp_n
\end{align*}
for $a\in L^2(\Om)$ and $t>0$. Then, as was proved in
\cite{GLY,Y2023}, we have the following lemma.

\begin{lem}\label{lem4}
{\rm(i)} Let $0<\ga<1$. Then, there exists a constant
$C=C(\ga)>0$ such that 
\[
\|A^\ga S(t)a\|\le C\,t^{-\al\ga}\|a\|,\quad
\|A^\ga K(t)a\|\le C\,t^{\al(1-\ga)-1}\|a\|
\]
for all $a\in L^2(\Om)$ and all $t>0$.

{\rm(ii)} Let $v\in L^2(0,T;\cD(A))$ satisfy
$v-a\in H_\al(0,T;L^2(\Om))$ and
$$
\pa_t^\al(v-a)=-A v+F,\quad t>0,
$$
with $a\in\cD(A^{\f12})$ and $F(t)=F(\,\cdot\,,t)\in L^2(0,T;L^2(\Om))$.
Then, $v$ allows the representation
$$
v(t)=S(t)a+\int^t_0K(t-s)F(s)\,\rd s,\quad t>0.
$$

{\rm(iii)} There holds
$$
\left\|\int^t_0K(t-s)F(s)\,\rd s\right\|_{H_\al(0,T;L^2(\Om))}
\le C\|F\|_{L^2(0,T;L^2(\Om))}.
$$
\end{lem}

{\bf Step 2.} Let $T>0$ be arbitrarily given. We show that there exists
$u\in L^2(0,T;L^2(\Om))$ such that $u\ge0$ in $\Om\times(0,T)$ and
\[
u(t)=S(t)a+\int^t_0K(t-s) u(s)^p\,\rd s,\quad0<t<T.
\]

Henceforth, by $C>0$, we denote generic constants depending on 
$\Om,T$, and $p$ but independent of the choices of functions
$a(x),u(x,t),v(x,t)$, etc.

Lemma \ref{lem4}(i) implies
\begin{equation}\label{4.2}
\|S(t)a\|\le C\|a\|,\quad t\ge0.
\end{equation}
We choose a constant $M>0$ sufficiently large such that 
\begin{equation}\label{4.3}
C(M+C\|a\|)^p\le M.
\end{equation}
Since $0<p<1$, we can easily verify the existence of such
$M>0$ satisfying \eqref{4.3}.

With this $M>0$, we define a set $\cB\subset L^2(0,T;L^2(\Om))$ by 
$$
\cB:=\{v\in L^2(0,T;L^2(\Om));v\ge0\mbox{ in }\Om\times(0,T),
\ \|v-S(t)a\|_{L^2(0,T;L^2(\Om))}\le M\}.
$$
We define a mapping $L$ by
$$
L v(t):=S(t)a+\int^t_0K(t-s)v^p(s)\,\rd s,\quad0<t<T,\quad v\in\cB.
$$
Now we will prove
\begin{equation}\label{4.4}
L\cB\subset\cB
\end{equation}
and
\begin{equation}\label{4.5}
L:\cB\longrightarrow\cB\mbox{ is a compact operator}.
\end{equation}

\begin{proof}[Proof of \eqref{4.4}]
Let $v\in\cB$. Then, we have
\begin{equation}\label{est-v}
\|v\|_{L^2(0,T;L^2(\Om))}
\le\|v-S(t)a\|_{L^2(0,T;L^2(\Om))}+\|S(t)a\|_{L^2(0,T;L^2(\Om))}
\le M+C\|a\|
\end{equation}
by the definition of $\cB$ and \eqref{4.2}.
On the other hand, Lemma \ref{lem3} implies
\begin{align*}
\|v^p\|_{L^2(0,T;L^2(\Om))}^2 & =\int_0^T\|v^p(t)\|^2\,\rd t
\le|\Om|^{1-p}\int_0^T\|v(t)\|^{2p}\,\rd t\\
& \le|\Om|^{1-p}T^{1-p}\left(\int_0^T\|v(t)\|^2\,\rd t\right)^p
=(|\Om|\,T)^{1-p}\|v\|_{L^2(0,T;L^2(\Om))}^{2p}.
\end{align*}
Therefore, substituting \eqref{est-v} into the above inequality yields
\begin{equation}\label{4.6}
\|v^p\|_{L^2(0,T;L^2(\Om))}
\le(|\Om|\,T)^{\f{1-p}2}\|v\|^p_{L^2(0,T;L^2(\Om))}
\le C(M+C\|a\|)^p.
\end{equation}
Consequently, Lemma \ref{lem4}(iii) and \eqref{4.3} and \eqref{4.6} imply
\begin{align*}
\|L v(t)-S(t)a\|_{L^2(0,T;L^2(\Om))}
& =\left\|\int^t_0K(t-s)v^p(s)\,\rd s\right\|_{L^2(0,T;L^2(\Om))}\\
& \le\left\|\int^t_0K(t-s)v^p(s)\,\rd s\right\|_{H_\al(0,T;L^2(\Om))}
\le C\|v^p\|_{L^2(0,T;L^2(\Om))}\\
& \le C(M+C\|a\|)^p\le M.
\end{align*}
Next, by $a\ge0$ in $\Om$ and $v\ge0$ in $\Om\times(0,T)$, we can
apply the comparison principle (e.g., \cite{LY}) to have
$$
\int^t_0K(t-s)v^p(s)\,\rd s\ge0\quad\mbox{in }\Om\times(0,T),
$$
and so, $Lv\ge0$ in $\Om\times(0,T)$. Hence, $L v\in\cB$, and thus the
proof of \eqref{4.4} is complete.
\end{proof}

\begin{proof}[Proof of \eqref{4.5}]
Since $S(t)a$ is a fixed element independent of $v$, it suffices to verify that
$$
L_0v(t):=\int^t_0K(t-s)v^p(s)\,\rd s
$$
is a compact operator from $\cB$ to $L^2(0,T;L^2(\Om))$.
Let $M_0>0$ be an arbitrarily chosen constant and let
$\|v\|_{L^2(0,T;L^2(\Om))}\le M_0$, $v\ge0$ in $\Om\times(0,T)$.
Then, Lemma \ref{lem3} indicates
\begin{equation}\label{4.7}
\|v^p\|_{L^2(0,T;L^2(\Om))}\le C\|v\|^p_{L^2(0,T;L^2(\Om))}
\le C M_0^p,
\end{equation}
with which we combine Lemma \ref{lem4}(iii) to obtain
\begin{equation}\label{4.8}
\|L_0v\|_{H_\al(0,T;L^2(\Om))}\le C M_0^p.
\end{equation}

Next, for small $\ve\in(0,1)$, in view of Lemma \ref{lem4}(i),
we estimate
\begin{align*}
\|A^\ve L_0v(t)\| & =\left\|\int^t_0A^\ve K(t-s)v^p(s)\,\rd s\right\|
\le\int_0^t\|A^\ve K(t-s)v^p(s)\|\,\rd s\\
& \le C\int^t_0(t-s)^{(1-\ve)\al-1}\|v^p(s)\|\,\rd s.
\end{align*}
Hence, in terms of \eqref{4.7}, Young's convolution inequality implies
\begin{align*}
\|A^\ve L_0v\|_{L^2(0,T;L^2(\Om))}
& \le C\left\|\int^t_0(t-s)^{(1-\ve)\al-1}\|v^p(s)\|\,\rd s\right\|_{L^2(0,T)}\\
& \le C\|t^{(1-\ve)\al-1}\|_{L^1(0,T)}
\left(\int^T_0\|v^p(t)\|^2\,\rd t\right)^{\f12}\\
& \le C\|v^p\|_{L^2(0,T;L^2(\Om))}\le C M_0^p.
\end{align*}
Since $\cD(A^\ve)\subset H^{2\ve}(\Om)$, we have
\begin{equation}\label{4.9}
\|L_0v\|_{L^2(0,T;H^{2\ve}(\Om))}
\le C\|A^\ve L_0v\|_{L^2(0,T;L^2(\Om))}\le C M_0^p.
\end{equation}
On the other hand, we know that the embedding
$L^2(0,T;H^{2\ve}(\Om))\cap H^\al(0,T;L^2(\Om))
\subset L^2(0,T;L^2(\Om))$ is compact
(e.g., Temam \cite[Theorem 2.1, p.\! 271]{Te}),
so that \eqref{4.8} and \eqref{4.9} imply that
$L_0:\cB\subset L^2(0,T;L^2(\Om))\longrightarrow\cB$ is compact.
This completes the proof of \eqref{4.5}. 
\end{proof}

Since $\cB$ is a closed and convex set in $L^2(0,T;L^2(\Om))$,
we can apply the Schauder fixed-point theorem to conclude that
$L$ possesses a fixed-point $u$ satisfying
\begin{equation}\label{4.10}
u(t)=S(t)a+\int^t_0K(t-s)u(s)^p\,\rd s,\quad0<t<T,
\quad u\ge0\quad\mbox{in }\Om\times(0,T).
\end{equation}

{\bf Step 3.} Recalling that $\f34<\ga\le1$, we note that if
$a\in H^{2\ga}(\Om)$ and $\pa_\nu a=0$ on $\pa\Om$,
then $a\in\cD(A^\ga)$. Now it remains to prove that
the fixed-point $u$ satisfies the regularity \eqref{1.2}.
To this end, we separate
$$
u(t)-a=(S(t)a-a)+\int^t_0K(t-s)u(s)^p\,\rd s
=:u_1(t)+u_2(t),\quad0<t<T.
$$

First, we verify \eqref{1.2} for $u_1(t)$. In the same way as that for
Yamamoto \cite[Lemma 5(i)]{Y2023}, we can prove that
$\pa_t^\al u_1(t)=-AS(t)a$ in $(0,T)$ and
$u_1\in H_\al(0,T;L^2(\Om))$ by $a\in\cD(A^\ga )$ with $\ga>\f34$.
Therefore, we obtain
$$
u_1\in H_\al(0,T;L^2(\Om)),\quad S(t)a\in L^2(0,T;\cD(A)).
$$

Next, we verify \eqref{1.2} for $u_2(t)$. In terms of
$u^p\in L^2(0,T;L^2(\Om))$, Lemma \ref{lem4}(iii) implies that
$u_2\in H_\al(0,T;L^2(\Om))$ and $\pa_t^\al u_2=-Au_2+u(t)^p$
for $0<t<T$. Therefore, we have $-A u_2\in L^2(0,T;L^2(\Om))$
or equivalently $u_2\in L^2(0,T;\cD(A))$.

Consequently, it is verified that the fixed-point $u$ satisfies \eqref{1.2}.
By \eqref{1.2} and \eqref{4.10} we see that
$u$ satisfies \eqref{1.1} in terms of \cite[Lemma 5]{Y2023}.
Thus, the proof of Theorem \ref{thm2} is complete.

\section{Concluding remarks and discussions}\label{sec-rem}

{\bf1.}
In this article, we consider the blowup exclusively in 
$L^1(\Om)$.  If we will discuss in the space $L^{\infty}(\Om)$, for example,
then we can more directly use a lower solution.  More precisely,
in \eqref{1.1} assuming that $\min_{x\in\overline\Om} a(x) 
=: a_1 > 0$, if we can find a function $g(t)$ satisfying
$$
\pa_t^\al (g(t) - a_1) \le g(t)^p, \quad 0<t<T, 
$$
then $\underline{u}(x,t) := g(t)$ for $x\in \Om$ and 
$0<t<T$ is a lower solution to \eqref{1.1}, i.e.,
$$
\begin{cases}
\pa_t^\al (\underline{u} - a_1) \le \tri \underline{u}
 + \underline{u}^p & \mbox{in }\Om\times(0,T),\\
\pa_{\nu}\underline{u} = 0 & \mbox{on }\pa\Om
\times (0,T).
\end{cases}
$$
Then, the comparison principle (e.g., \cite{LY}) yields
$$
g(t) \le u(x,t), \quad x\in \Om, \, 0<t<T.
$$

As $g(t)$, we take a similar function to \eqref{def-lowsol}:
$$
g(t) := a_1\left( \f{T}{T-t}\right)^m, \quad m\in \BN.
$$
Then, by \eqref{eq-calc2} we have
$$
\pa_t^\al (g(t)-a_1) \le \f{a_1m}{T^{\al}\Ga(2-\al)}
\left( \f{T}{T-t}\right)^{m+1}, \quad 0<t<T.
$$
Therefore, for $mp - (m+1) \ge 0$, it suffices to choose $T>0$ such that 
$$
 \f{a_1m}{T^{\al}\Ga(2-\al)}
\left( \f{T}{T-t}\right)^{m+1} \le a_1^p \left( \f{T}{T-t}\right)^{mp}
= g(t)^p,  \quad 0<t<T,
$$
i.e.,
$$
 \f{a_1^{1-p}m}{T^{\al}\Ga(2-\al)}
\le \xi^{mp-(m+1)} \quad \mbox{for all $\xi \ge 1$}
$$
by setting $\xi := \f{T}{T-t} \ge 1$. Hence, $g(t)$ is a lower solution
if 
$$
 \f{a_1^{1-p}m}{T^{\al}\Ga(2-\al)} \le 1, \quad
\mbox{i.e.}, \quad 
T \ge \left( \f{a_1^{1-p}m}{\Ga(2-\al)}\right)^{\f1{\al}}
$$
for $mp \ge m+1$.  Choosing the minimum $m\in\BN$ and arguing similarly to 
the final part of the proof of Theorem \ref{thm1}, we obtain an inequality for 
the blowup time $T_{\al,p,a}(\infty)$ in $L^{\infty}(\Om)$:
\begin{equation}\label{5.1}
T_{\al,p,a}(\infty) 
\le \left( \f{ \left[ \f1{p-1} \right] + 1}
{\Ga(2-\al)a_1^{p-1}} \right)^{\f1{\al}}
=: T^*_{\infty}(\al,p,a),
\end{equation}
where $[q]$ denotes the maximum natural number not exceeding $q>0$.

We compare $T_{\infty}^*(\al,p,a)$ with an upper bound $T^*(\al,p,a)$ of 
the blowup time in $L^1(\Om)$.  Noting that 
$a_1 \le \f1{| \Om|}\int_{\Om} a(x)\, \rd x$, we can interpret 
that $\f1{|\Om|}\int_{\Om} a(x)\, \rd x$ is comparable with 
$a_1$ and so we consider the case where $a_1 = \f1{|\Om|}
\int_{\Om} a(x)\, \rd x$. Then, by \eqref{1.4}, we have
\begin{equation}\label{5.2}
T_{\al,p,a} \le T^*(\al,p,a) 
= \left( \f{\f1{p-1}}{\Ga(2-\al)a_1^{p-1}}
\right)^{\f1{\al}}.
\end{equation}
Hence, $\left[ \f1{p-1}\right] + 1 \ge \f1{p-1}$ implies
$T^*(\al,p,a) < T_{\infty}^*(\al,p,a)$.

To sum up, for the $L^1(\Om)$-blowup time $T_{\al,p,a}$ and
the $L^{\infty}(\Om)$-blowup time $T_{\al,p,a}(\infty)$,
our upper bounds $T^*(\al,p,a)$ and $T^*_{\infty}(\al,p,a)$ of
$T_{\al,p,a}$ and $T_{\al,p,a}(\infty)$ are given by \eqref{5.2} 
and \eqref{5.1} respectively.
Although we should expect $T_{\infty}^*(\al,p,a) \le T^*(\al,p,a)$ by means
$T_{\al,p,a}(\infty) \le T_{\al,p,a}$, which follows from
$\|u(\,\cdot\,,t)\|_{L^1(\Om)}
\le C\|u(\,\cdot\,,t)\|_{L^{\infty}(\Om)}$, but our bounds do not satisfy.
The upper bound depends on our choice of lower solutions, and
it is a future work to discuss sharper bounds.\medskip

{\bf2.}
Restricting the nonlinearity to the polynomial type $u^p$, in this article,
we investigate semilinear time-fractional diffusion equations
with the homogeneous Neumann boundary condition.
With nonnegative initial values, we obtained the blowup of solutions 
with $p>1$
as well as the global-in-time existence of solutions with $0<p<1$.
The key ingredient for the latter is the Schauder fixed-point theorem,
whereas that for the former turns out to be a comparison principle
for time-fractional ordinary differential equations
(see Lemma \ref{lem2}) and the construction of a lower solution
of the form \eqref{def-lowsol}. We can similarly discuss
the blowup for certain semilinear terms  like the exponential type
$\mathrm e^u$ and some coupled systems. More generally,
it appears plausible to consider a general convex semilinear term $f(u)$,
which deserves further investigation.

Technically, by introducing
\[
\eta(t):=\int_\Om u(x,t)\,\rd x=(u(\,\cdot\,,t),1)_{L^2(\Om)},
\]
we reduce the blowup problem to the discussion of a time-fractional
ordinary differential equation. As was mentioned in Remark \ref{rmk1},
indeed $1$ is the eigenfunction for the smallest eigenvalue $0$ of
$-\tri$ with $\pa_\nu u=0$. On this direction,
it is not difficult to replace $-\tri$ with a more general elliptic operator.
Actually, in place of $1$, one can choose an eigenfunction $\vp_1$
for the smallest eigenvalue $\la_1$ and consider
$\eta(t):=(u(\,\cdot\,,t),\vp_1)_{L^2(\Om)}$ to follow the above
arguments. In this case, it is essential that $\la_1\ge0$ and $\vp_1$
does not change sign. We can similarly discuss the homogeneous
Dirichlet boundary condition.\medskip

{\bf3}.
In the proof of Theorem \ref{thm1}, we obtained
an upper bound $T^*(\al,p,a)$ of the blowup time $T$ (see \eqref{1.4}),
but there is no guarantee for its sharpness. Sharp estimates for
the blowup time in the time-fractional case is expected to be more
complicated than the parabolic case, which is postponed to a future topic.

We briefly investigate the monotonicity of
$$
T^*(\al,p,a) =
\left( \f1
{(p-1)\Ga(2-\al)\left( \f1{| \Om|}
\int_{\Om} a(x)\, \rd x \right)^{p-1} } \right)^{\f1{\al}}>0
$$
as a function of $\al \in (0,1)$ with fixed $p$ and $a$. Setting
$$
C_{p,a}:= (p-1)\left( \f1{|\Om|}
\int_{\Om} a(x)\, \rd x \right)^{p-1}>0,
$$
we can verify that there exist positive constants $C^*\ge1$ and $C_*\le1$ such that 
$T^*(\al,p,a)$ is monotone increasing in $\al$ if $C_{p,a}\ge C^*$ and 
monotone decreasing in $\al$ if $C_{p,a}\le C_*$.

Indeed, setting $f(\al):= T^*(\al,p,a)$ for simplicity for fixed 
$p$ and $a$, we have
$$
\log f(\al) = -\f1{\al}\log (C_{p,a}\Ga(2-\al)),
$$
i.e.,
$$
\f{f'(\al)}{f(\al)} = -\f1{\al}
\f{\f{\rd}{\rd\al}(\Ga(2-\al))}{\Ga(2-\al)}
+ \f1{\al^2}\log (C_{p,a}\Ga(2-\al))
= \f1{\al^2}\left( \log(C_{p,a}\Ga(2-\al))
+ \al \f{\Ga'(2-\al)}{\Ga(2-\al)}\right)
$$
for $0<\al<1$.  We set $\delta_0:= \min_{0\le\al\le 1}
\Ga(2-\al) > 0$ and
$M_1:= \max_{0\le\al\le1}|\f{\Ga'(2-\al)}{\Ga(2-\al)}|$.
Then,
$$
\f{f'(\al)}{f(\al)} \ge \f1{\al^2}
(\log (C_{p,a}\delta_0) - M_1) > 0
$$
if $C_{p,a}>0$ is sufficiently large. On the other hand,
since $\Ga(2-\al)\le1$ for $0\le\al\le1$, we see that 
$$
\f{f'(\al)}{f(\al)} \le \f1{\al^2}
(\log C_{p,a} + \al M_1) \le \f1{\al^2}(\log C_{p,a} + M_1) < 0
$$
if $C_{p,a}>0$ is sufficiently small. 

Since 
$$
\log (\Gamma(2-\al)^{-\frac{1}{\al}})
= \frac{\log \Gamma(2-\al) - \log \Gamma(2)}{-\alpha}\,
\longrightarrow \, \frac{d}{d\beta}\Gamma(\beta)\vert_{\beta=2}
$$
as $\alpha \to 0+$, we have $\lim_{\al\to 0+} f(\al) = e^{\Gamma'(2)}$
if $C_{p,a} = 1$.  Therefore,
\[
\lim_{\al\to0+}f(\al)=\begin{cases}
+\infty, & C_{p,a}<1,\\
\mathrm e^{\Ga'(2)}, & C_{p,a}=1,\\
0, & C_{p,a}>1.
\end{cases}
\]
In particular, $f(\al)$ cannot be monotone increasing for $C_{p,a}<1$
and cannot be monotone decreasing for $C_{p,a}>1$,
which implies $C^*\ge1$ and $C_*\le1$.\medskip

\renewcommand{\labelenumi}{(\roman{enumi})}
{\bf4}.
Related to the blowup, we should study the following issues:
\begin{enumerate}
\item Lower bounds or characterization of the blowup times.
\item Asymptotic behavior or lower bound of a solution near the 
blowup time.
\item Blowup set of a solution $u(x,t)$, which means the set of $x\in\Om$,
where $|u(x,t)|$ tends to $\infty$ as $t$ approaches the blowup time.
\end{enumerate}
For $\al=1$, comprehensive and substantial works have been accomplished.  
We are here restricted to refer to Chapter II of Quittner and Souplet \cite{QS}
and the references therein.
However, for $0<\al<1$, by the memory effect of 
$\partial_t^{\al}u(\,\cdot\,,t)$ which involves the past value of $u$,
several useful properties for discussing the above issues (i)--(iii) do not 
hold.
Thus, the available results related to the blowup are still 
limited for $0<\al<1$, and it is up to future studies to pursue (i)--(iii).
\bigskip

{\bf Acknowledgements:} This work was completed during the third author's stay at
Sapienza Universit\`a di Roma in January and February 2023.
The authors thank the anonymous referees for careful reading and valuable comments.\bigskip

{\bf Funding information:} This work is supported by MUR\_PRIN 201758MTR2\_003
``Direct and inverse problems for partial differential equations: theoretical aspects and applications''.
The Istituto Nazionale di Alta Matematica (IN$\delta$AM) and the
``Gruppo Nazionale per l'Analisi Matematica, la Probabilit\'a e le loro Applicazioni'' (GNAMPA)
also supported the authors, in particular, in the organization of the GNAMPA Workshop
``Recent advances in direct and inverse problems for PDEs and applications''
(Sapienza Universit\`a di Roma, December 5--7, 2022).
The first author is supported by the French-German-Italian Laboratoire International Associ\'e (LIA), named
COPDESC, on Applied Analysis, issued by CNRS, MPI, and IN$\delta$AM.
The second author is supported by Grant-in-Aid for Early-Career Scientists 22K13954 from
Japan Society for the Promotion of Science (JSPS).
The third author is supported by Grants-in-Aid for Scientific Research (A) 20H00117 and
Grant-in-Aid for Challenging Research (Pioneering) 21K18142, JSPS.
The second and the third authors are supported by Fund for the
Promotion of Joint International Research (International Collaborative Research) 23KK0049, JSPS.
The third author was also both INdAM visiting professor and GNAMPA visiting professor in 2022.\bigskip

{\bf Conflict of interest:} The authors state no conflict of interest.


\end{document}